\documentclass[pdflatex,sn-mathphys-num]{sn-jnl}

\usepackage{graphicx}%
\usepackage{multirow}%
\usepackage{amsmath,amssymb,amsfonts}%
\usepackage{amsthm}%
\usepackage{mathrsfs}%
\usepackage[title]{appendix}%
\usepackage{xcolor}%
\usepackage{textcomp}%
\usepackage{manyfoot}%
\usepackage{booktabs}%
\usepackage{algorithm}%
\usepackage{algorithmicx}%
\usepackage{algpseudocode}%
\usepackage{listings}%
\theoremstyle{thmstyleone}%
\newtheorem{theorem}{Theorem}
\newtheorem{lemma}[theorem]{Lemma}%
\newtheorem{assumption}[theorem]{Assumption}%

\theoremstyle{thmstyletwo}%
\newtheorem{example}{Example}%
\newtheorem{remark}{Remark}%

\theoremstyle{thmstylethree}%
\newtheorem{definition}{Definition}%

\begin{document}

\title[Article Title]{An a Posteriori Error Estimator for $C^0$ IPG Approximation of Multiple and Clustered Eigenvalues of the Biharmonic Operator}

%%===========================================================

\author[1]{\fnm{Jianing} \sur{Guo}}\email{2111167@tongji.edu.cn}

\author*[1]{\fnm{Qigang} \sur{Liang}}\email{qigang$\_$liang@tongji.edu.cn}

\affil[1]{\orgdiv{School of Mathematical Sciences, Tongji University, Shanghai, 200092, China and Key Laboratory of Intelligent Computing and Applications (Tongji University), Ministry of Education, Shanghai, 200092, China}}

\abstract{In this paper, based on the $C^0$ interior penalty Galerkin ($C^0$IPG) discretization, we propose and analyze an a posteriori error estimator for eigenfunctions associated with multiple and clustered eigenvalues of the biharmonic operator. The proposed estimator may capture the local singularities of the target eigenfunctions efficiently and play an important role in adaptive procedures. We develop a
rigorous cluster-projection-based analysis and introduce an auxiliary theoretical
error estimator to connect the invariant subspace error with a computable residual
estimator. Most importantly, the resulting reliability and efficiency bounds
are robust with respect to the mesh size, the mesh level and the internal spectral gaps
within the target cluster. Numerical experiments support the theoretical
results and demonstrate the robustness of the proposed estimator.
}


\maketitle

\section{Introduction}\label{sec1}
\par Fourth-order eigenvalue problems arise naturally in the mathematical modeling of thin-plate bending, elastic vibration, phase-field separation, and high-order diffusion problems. A typical model is the clamped biharmonic eigenvalue problem: find a pair $(\lambda,u)\in \mathbb{R}^+\times H^2_0(\Omega)$ with $u\neq 0$ such that
\begin{equation}\label{bi-Intro}
\int_\Omega D^2u:D^2v\,dx=\lambda\int_\Omega uv\,dx,\quad \forall v\in H^2_0(\Omega),
\end{equation}
where $\Omega\subset\mathbb{R}^2$ is a bounded polygonal domain. The finite element approximation of self-adjoint eigenvalue problems has been systematically studied for several decades, with fundamental theory established by Babu\v{s}ka and Osborn \cite{BabuskaOsborn1989, BabuskaOsborn1991} and further developments reviewed by Boffi \cite{Boffi2010}. However, the numerical discretization of fourth-order eigenvalue problems remains challenging, owing to the $C^1$-continuity requirement of the underlying variational formulation and the limited regularity of eigenfunctions on nonsmooth domains; see \cite{MR595625, MR775683}.

 \par For biharmonic boundary value problems, various finite element methods have
been developed, including conforming, nonconforming, mixed, and discontinuous
Galerkin; see \cite{Morley1968,CiarletRaviart1974,Ciarlet1978,Engel2002}. 
Subsequently, a posteriori error analysis has been carried out for several of these approaches, including Morley and related nonconforming finite elements \cite{MR2291934,hushi2009,cahu2013}, mixed methods \cite{gudi2011}, discontinuous Galerkin methods \cite{Geo2011,dong2021}, and conforming Kirchhoff plate elements \cite{gus2018}. The $C^0$IPG method has also been widely studied for biharmonic problems. It is based on the Lagrange finite elements and uses penalty terms to weakly enforce the continuity of normal derivatives across interelement boundaries. For lower-order elements, the $C^0$IPG method has a computational complexity comparable to that of classical nonconforming methods, while higher-order spaces are much easier to construct than conforming $C^1$ finite element spaces. Brenner and Sung \cite{BrennerSung2005} developed a systematic analysis of this method for fourth-order elliptic boundary value problems on polygonal domains, and a posteriori error estimates were established in \cite{Brenner2010, Brenner2012}.

\par The a posteriori analysis for elliptic eigenvalue problems has also been extensively developed, especially for multiple and clustered eigenvalues. For second-order elliptic eigenvalue problems, a substantial body of work has established a mature framework for simple, multiple, and clustered eigenvalues; see, for example, \cite{daixy, Gallistl2014, Gallistl2015, Maday2017, Maday2018, Maday2020, liu2022,Li2026,li2025}. Dai et al.~\cite{daixy} studied adaptive finite element approximations for multiple eigenvalues and established convergence rates and quasi-optimal complexity. Gallistl \cite{Gallistl2014, Gallistl2015} developed a systematic adaptive framework for eigenvalue clusters, together with corresponding a posteriori error estimates. Canc\`es et al.~\cite{Maday2017, Maday2018, Maday2020} established guaranteed and robust a posteriori bounds for eigenvalues and eigenvectors, covering both multiplicities and clusters. Fully computable bounds for eigenfunctions were subsequently derived by Liu and Vejchodsk\'y \cite{liu2022}, while Li et al.~\cite{li2025,Li2026} recently developed pointwise a posteriori error estimators for simple, multiple, and clustered eigenvalues. These contributions provide important tools for the a posteriori analysis of eigenspace approximations for biharmonic eigenvalue problems.

\par For biharmonic eigenvalue problems, nonconforming Morley finite element approximations, including adaptive methods and a posteriori error estimates, have been extensively studied \cite{MR3407244, MR3315683, MR3771503}. The $C^0$ interior penalty Galerkin approximation was analyzed in \cite{brennereig}, while a priori and a posteriori error estimates for mixed discontinuous Galerkin approximations were derived in \cite{WangXiongWuLuo2019}. In the $C^0$IPG setting, Li and Yang \cite{LiYang2018b} developed residual-based adaptive algorithms for simple eigenvalue of the biharmonic operator. However, multiple and clustered eigenvalues remain more challenging and arise frequently in practical applications. The design and analysis of the a posteriori error estimator for multiple and clustered eigenvalues have not yet been developed. We must emphasize that it is nontrivial to give a rigorous a posteriori analysis for eigenfunctions corresponding to interior multiple and clustered eigenvalues. First, the error must be measured at the eigenspace level through a suitable cluster projection, rather than through individual eigenfunctions. Second, the analysis must control the consistency errors caused by the weak continuity of normal derivatives across interelement boundaries. Finally, an auxiliary theoretical estimator must be constructed to bridge the invariant subspace error and the computable residual estimator while avoiding any dependence on the internal spectral gaps.

\par  In this paper, based on the $C^0$ interior penalty Galerkin discretization, we propose and analyze an a posteriori error estimator for eigenfunctions associated with multiple and clustered eigenvalues of the biharmonic operator. The analysis is based on the cluster
projection $\Lambda_h:=P_h\circ R_h,$ where $R_h$ denotes the quasi-Ritz projection and $P_h$ is the $L^2$-orthogonal projection onto the discrete cluster space. Building on this projection, and with the aid of several analytical tools, including the properties of the enriching operator, an auxiliary theoretical error estimator, and a consistency decomposition adapted to the $C^0$IPG approximation, we derive the following reliability and efficiency estimates:
\[
\|u_i-\Lambda_hu_i\|_h
\le
C_{1}\eta_h,
\]
and
\[
C_2 \eta_h
\le
\sum_{i\in J}\|u_i-\Lambda_hu_i\|_h,
\]
where $i\in J:=\{m,m+1,\ldots,M\},$ $(\lambda_i,u_i)$ is the $i$-th eigenpair of \eqref{bi-Intro}, $\|\cdot\|_h$ is the mesh-dependent energy norm, and $C_{1}$ and $C_{2}$ are positive constants independent of the mesh size, the mesh level and the gaps among the target eigenvalues within the cluster. This property makes the estimator robust for tightly clustered eigenvalues. Finally, numerical results verify our theoretical findings. 

\par The rest of this paper is organized as follows. In Section 2, we introduce the notation and preliminary tools for the subsequent analysis. Section 3 develops the cluster projection framework and introduces the theoretical and computable error estimators. It also proves the main results on reliability, equivalence, and efficiency. We present numerical experiments in Section 4 and conclusions in Section 5.

\section{Preliminaries}\label{sec:setting}
\par We consider the clamped biharmonic eigenvalue problem
\begin{equation}\label{eq:strong-eig}
\begin{cases}
\Delta^2 u=\lambda u & \text{in }\Omega,\\
u=0,\quad \dfrac{\partial u}{\partial n}=0 & \text{on }\partial\Omega,
\end{cases}
\end{equation}
where $\Omega\subset\mathbb{R}^2$ is a bounded polygonal domain.
Its weak formulation is to find $(\lambda,u)\in\mathbb{R}^{+}\times H^2_0(\Omega)$ with $u\neq 0$ such that
\begin{equation}\label{eq:cont-eig}
a(u,v)=\lambda b(u,v)\qquad\forall v\ \in H^2_0(\Omega),
\end{equation}
where
\[
a(u,v):=\int_\Omega D^2u:D^2v\,dx,
\qquad
b(u,v):=\int_\Omega uv\,dx.
\]
Throughout the paper, eigenfunctions are normalized by $\|u\|_{L^2(\Omega)}^2 = b(u,u)=1$. The bilinear form $a(\cdot,\cdot)$ defines an inner product on $H^2_0(\Omega)$, and the induced norm is equivalent to the usual $H^2$-seminorm $|\cdot|_{H^2({\Omega})}$. 
Consequently, the eigenvalues of \eqref{eq:cont-eig} can be arranged as
\[
0<\lambda_1\le \lambda_2\le \cdots \le \lambda_j\le \cdots,
\]
where each eigenvalue is repeated according to its multiplicity and $\lambda_j\to +\infty$ as $j\to+\infty$.
Let $J:=\{m,m+1,\cdots,M\}$ be the index set of the target eigenvalues,
\[
\lambda_{m-1}<\lambda_m\le\cdots\le\lambda_M<\lambda_{M+1},
\]
with the convention $\lambda_0:=0$. The corresponding eigenspace is defined by $W:=\operatorname{span}\{u_j:j\in J\}.$

We now introduce the mesh notation and the finite element space used in the $C^0$IPG discretization. Let $\mathcal{T}_0$ be an initial quasi-uniform and
shape-regular triangulation of $\Omega$ with mesh size $h_0$, and let
$\mathcal{T}_h$ be a shape-regular triangulation obtained from
$\mathcal{T}_0$ after a finite number of adaptive refinements. Denote by
$\mathcal{E}_h$ the set of all edges of $\mathcal{T}_h$, with the decomposition
\[
\mathcal{E}_h=\mathcal{E}_h^I\cup\mathcal{E}_h^D,
\]
where $\mathcal{E}_h^I$ and $\mathcal{E}_h^D$ are the sets of interior and
boundary edges, respectively. For $T\in\mathcal{T}_h$ and
$e\in\mathcal{E}_h$, set $h_T:=\operatorname{diam}(T)$ and $h_e:=|e|.$ Let $V_h\subset H^1_0(\Omega)$ be the Lagrange finite element space of degree $k\ge 2$ associated with $\mathcal{T}_h$. For any piecewise $H^2$ function
$v$, the broken Hessian $D_h^2v$ is defined elementwise.

Next, we present the jump and average operators appearing in the $C^0$IPG
formulation and the residual estimators. For an interior edge
$e=\partial T^+\cap\partial T^-\in\mathcal{E}_h^I$, let $\boldsymbol n_e$
be the unit normal vector pointing from $T^-$ to $T^+$, and let
$\boldsymbol t_e$ be a unit tangent vector along $e$. We set
\[
\partial_{nn}v:=\boldsymbol n_e^T D^2v\,\boldsymbol n_e,
\qquad
\partial_{nt}v:=\boldsymbol n_e^T D^2v\,\boldsymbol t_e .
\]
The jumps and averages on $e$ are then defined by
\[
\begin{aligned}
\left[\!\left[\partial_n v\right]\!\right]
&:=
\nabla v|_{T^+}\cdot\boldsymbol n_e
-
\nabla v|_{T^-}\cdot\boldsymbol n_e,
&
\left\{\!\!\left\{\partial_{nn}v\right\}\!\!\right\}
&:=
\frac12
\left(
\partial_{nn}v|_{T^+}
+
\partial_{nn}v|_{T^-}
\right),
\\
\left[\!\left[\partial_{nn}v\right]\!\right]
&:=
\partial_{nn}v|_{T^+}
-
\partial_{nn}v|_{T^-},
&
\left[\!\left[\partial_n(\Delta v)\right]\!\right]
&:=
\nabla(\Delta v)|_{T^+}\cdot\boldsymbol n_e
-
\nabla(\Delta v)|_{T^-}\cdot\boldsymbol n_e.
\end{aligned}
\]
For a boundary edge $e\in\mathcal{E}_h^D$ with outward unit normal
$\boldsymbol n_e$, we set
\[
\left[\!\left[\partial_n v\right]\!\right]
:=
-\nabla v\cdot\boldsymbol n_e,
\qquad
\left\{\!\!\left\{\partial_{nn}v\right\}\!\!\right\}
:=
\partial_{nn}v.
\]

Following \cite{Brenner2012,brennereig}, the $C^0$ interior penalty bilinear form is given by
\begin{align*}
a_h(w_h,v_h)
:= {}&
\sum_{T\in\mathcal{T}_h}\int_T D^2w_h:D^2v_h\,dx
+\sum_{e\in\mathcal{E}_h}\int_e
\left\{\!\!\left\{\partial_{nn}w_h\right\}\!\!\right\}
\left[\!\left[\partial_n v_h\right]\!\right]\,ds
\notag\\
&+\sum_{e\in\mathcal{E}_h}\int_e
\left\{\!\!\left\{\partial_{nn}v_h\right\}\!\!\right\}
\left[\!\left[\partial_n w_h\right]\!\right]\,ds
+\sigma\sum_{e\in\mathcal{E}_h}h_e^{-1}\int_e
\left[\!\left[\partial_n w_h\right]\!\right]
\left[\!\left[\partial_n v_h\right]\!\right]\,ds,
\end{align*}
where $\sigma>0$ is the penalty parameter. For sufficiently large $\sigma$, the bilinear form $a_h(\cdot,\cdot)$ is symmetric and coercive on $V_h$. The associated mesh-dependent energy norm is defined by
\begin{equation*}
\|v_h\|_h^2
:=
\sum_{T\in\mathcal{T}_h}|v_h|_{H^2(T)}^2
+\sigma\sum_{e\in\mathcal{E}_h}h_e^{-1}
\left\|
\left[\!\left[\partial_n v_h\right]\!\right]
\right\|_{L^2(e)}^2.
\end{equation*}

\par With the bilinear form $a_h(\cdot,\cdot)$ defined above, the $C^0$IPG approximation of \eqref{eq:cont-eig} reads as follows: find $(\lambda_{h,j},u_{h,j})\in\mathbb{R}\times V_h$ with $u_{h,j}\neq 0$ such that
\begin{equation}\notag
a_h(u_{h,j},v_h)=\lambda_{h,j}b(u_{h,j},v_h)\qquad\forall v_h\in V_h,
\end{equation}
together with the normalization
\[
b(u_{h,j},u_{h,k})=\delta_{jk}.
\]
We denote by
\[
W_h:=\operatorname{span}\{u_{h,j}\mid j\in J\}
\]
the corresponding discrete eigenspace.

For the analysis of multiple and clustered eigenvalues, we use a spectral separation condition in the spirit of the cluster projection framework in \cite{Gallistl2014, Gallistl2015}. 

\begin{assumption}\label{ass:sep}
Assume that 
\begin{equation}\label{mhj}
M_{h,J}:=\max_{\substack{1\le j\le \dim V_h\\ j\notin J}}
\max_{i\in J}
\frac{\lambda_i}{|\lambda_{h,j}-\lambda_i|}<\infty.
\end{equation}
\end{assumption}

We next recall two auxiliary tools that will be used repeatedly in the a posteriori analysis. The first is the standard local interpolation estimate for the Lagrange nodal interpolation operator. Throughout the paper, $C$ (with or without subscripts) denotes a generic positive constant, independent of the mesh size, the mesh level and the gaps among the target eigenvalues within the
cluster, which may take different values at different occurrences.
\begin{lemma}[Interpolation estimate {\cite{Ciarlet1978,Brenner2010}}]\label{lem:interp}
Let $\Pi_h:C(\overline{\Omega})\to V_h$ denote the Lagrange nodal interpolation operator. For any $\zeta\in H^2(\omega_T)\cap C(\overline{\omega_T})$, there exists a constant $C>0$, independent of the mesh size, such that
\begin{equation}\notag
h_T^{-4}\|\zeta-\Pi_h\zeta\|_{L^2(T)}^2
+h_T^{-2}|\zeta-\Pi_h\zeta|_{H^1(T)}^2
+|\zeta-\Pi_h\zeta|_{H^2(T)}^2
\le C|\zeta|_{H^2(\omega_T)}^2,
\end{equation}
where $\omega_T$ denotes the element patch associated with $T$.
\end{lemma}

The second tool is an enriching operator, which transfers discrete functions into the conforming space $H^2_0(\Omega)$ while controlling the nonconformity by the jump terms in the penalty norm.

\begin{lemma}[Enriching operator {\cite{Brenner2010,BrennerSung2005}}]\label{lem:enriching}
There exists an enriching operator $E_h:V_h\to H^2_0(\Omega)$ and a constant $C>0$, independent of the mesh size, such that for all $v_h\in V_h$,
\begin{align*}
    \sum_{T\in\mathcal{T}_h}
\left(
h_T^{-4}\|v_h-E_hv_h\|_{L^2(T)}^2
+h_T^{-2}|v_h-E_hv_h|_{H^1(T)}^2
\right)
&\le
C\sum_{e\in\mathcal{E}_h}h_e^{-1}
\left\|\left[\!\left[\partial_n v_h\right]\!\right]\right\|_{L^2(e)}^2,
\end{align*}
and
\begin{align*}
 \sum_{T\in\mathcal{T}_h}|v_h-E_hv_h|_{H^2(T)}^2
&\le
C\sum_{e\in\mathcal{E}_h}h_e^{-1}
\left\|\left[\!\left[\partial_n v_h\right]\!\right]\right\|_{L^2(e)}^2.
\end{align*}
\end{lemma}

\section{An a posteriori error estimator} 
\par In this section, we first state the main theorem, which establishes the reliability of the computable estimator for the approximation of the whole target eigenspace. The proof is divided into two parts. In Section 3.1, we derive the reliability and efficiency properties of the theoretical error estimator. In Section 3.2, we prove the equivalence between the theoretical estimator and the computable estimator. The proof of the main theorem is then completed by combining these results.

\par For each discrete eigenpair $(\lambda_{h,j},u_{h,j})$, we define the local computable indicators by
\begin{align*}
\eta_T(\lambda_{h,j},u_{h,j})&:=h_T^2\|\lambda_{h,j}u_{h,j}-\Delta^2u_{h,j}\|_{L^2(T)},\\
\eta_{e,1}(\lambda_{h,j},u_{h,j})&:=\sigma^{1/2} h_e^{-1/2}\left\|\left[\!\left[\partial_n u_{h,j}\right]\!\right]\right\|_{L^2(e)},\\
\eta_{e,2}(\lambda_{h,j},u_{h,j})&:=h_e^{1/2}\left\|\left[\!\left[\partial_{nn}u_{h,j}\right]\!\right]\right\|_{L^2(e)},\\
\eta_{e,3}(\lambda_{h,j},u_{h,j})&:=h_e^{3/2}\left\|\left[\!\left[\partial_n(\Delta u_{h,j})\right]\!\right]\right\|_{L^2(e)}.
\end{align*}
The global computable estimator is obtained by summing these local contributions over all discrete eigenfunctions in the cluster and over all elements, as follows:
\begin{equation}\label{compueta}
\eta_h^2:=\sum_{j\in J}\left(\sum_{T\in\mathcal{T}_h}\eta_T^2(\lambda_{h,j},u_{h,j})+\sum_{e\in\mathcal{E}_h}\eta_{e,1}^2(\lambda_{h,j},u_{h,j})+\sum_{\ell=2,3}\sum_{e\in\mathcal{E}_h^I}\left(\eta_{e,\ell}^2(\lambda_{h,j},u_{h,j})\right)\right).
\end{equation}

To relate this computable estimator to the eigenspace error, we introduce the following projection operators. Given $ f \in L^2(\Omega)$, let $w\in V:=H_0^2(\Omega)$ denote the unique solution to the linear problem
\[
a(w,v) = b(f,v)\qquad \forall v \in V.
\]
The quasi-Ritz projection $R_hw\in V_h$ is defined by
\[
a_h(R_h w,v_h)=b(f,v_h)
\qquad
\forall v_h\in V_h.
\]
Let $P_h:L^2(\Omega)\to W_h$ denote the $L^2$-orthogonal projection onto the discrete cluster space. The associated cluster projection is then given by
\[
\Lambda_h:=P_h\circ R_h.
\]
The corresponding theoretical error estimator will be defined in terms of $\Lambda_h$ and will serve as the main tool in the reliability analysis. We also introduce the solution operator
$T:L^2(\Omega)\to V:=H_0^2(\Omega)$ satisfying
\[
a(Tf,v)=b(f,v)
\qquad
\forall v\in V.
\]
The corresponding approximation parameter is defined by
\begin{equation}\label{eq:rhoh-def}
\rho_h
:=
\sup_{\substack{f\in L^2(\Omega)\\ \|f\|_{L^2(\Omega)}=1}}
\inf_{v_h\in V_h}
\|Tf-v_h\|_h.
\end{equation}
Then we are in a position to state the main result of this paper.

\begin{theorem}\label{thm:main-reliability}
Let $(\lambda_i,u_i)$ be an exact eigenpair belonging to the target cluster. Assume that the initial mesh size $h_0$ is sufficiently small. Then
\begin{align*}
\|u_{i}-\Lambda_hu_{i}\|_{h}
&\le
C\left(\left(\frac{\lambda_{M+1}}{\lambda_m}\right) \eta_h+\lambda_{i}(1+M_{h,J})\rho_h\|u_{i}-\Lambda_hu_{i}\|_{h}\right),\ \ \ i\in J:=\{m,m+1,...,M\},\\
     \eta_h
&\le C \sqrt{(2N+4N^2)}\sum_{i\in J}
\left(
1+\lambda_i(1+M_{h,J})\rho_h
\right)
\|u_i-\Lambda_hu_i\|_h,
\end{align*}
where $N := \operatorname{card}(J)$, $\|\cdot\|_{h}$ is the mesh-dependent energy norm, $\rho_h\ (\rho_h\to 0,$ as $h\to0)$ is the approximation parameter of the solution operator defined in \eqref{eq:rhoh-def} and $M_{h,J}$ is independent of the internal gaps among the target eigenvalues $\{\lambda_{i}\}_{i\in J}$ defined in \eqref{mhj}.
\end{theorem}

\begin{remark}\label{rem:constants-main}

For sufficiently small initial mesh size $h_0$ such that
\[
C\lambda_i(1+M_{h,J})\rho_h\le \frac12 \qquad \text{and}\qquad \max_{i\in J}
\left(1+\lambda_i(1+M_{h,J})\rho_h\right)
\le 2,
\]
 Theorem~\ref{thm:main-reliability} implies
\[
\|u_i-\Lambda_hu_i\|_h
\le C_1\eta_h,
\qquad i\in J,
\]
and
\[
C_2\eta_h
\le
\sum_{i\in J}\|u_i-\Lambda_hu_i\|_h,
\]
where
\[
C_1=2C\frac{\lambda_{M+1}}{\lambda_m},
\qquad
C_2=
\left(
2C\sqrt{2N+4N^2}
\right)^{-1}.
\]
Both $C_1$ and $C_2$ are positive constants independent of the mesh size and of the internal eigenvalue gaps in the target cluster.
\end{remark}

\subsection{Reliability and efficiency of the theoretical error estimator}\label{3.1}
In this subsection, we introduce the theoretical error estimator and establish its reliability and efficiency in the mesh-dependent energy norm $\|\cdot\|_h$. Although this estimator is not computable, it is defined in terms of the fixed exact eigenspace through the operators $P_h$ and $\Lambda_h$, and is well suited to the analysis of multiple and clustered eigenvalues. We begin by introducing two properties of the cluster projection operator $\Lambda_h$. The first one is a discrete identity that plays a key role in the residual representation.
\begin{lemma}\label{lambda_h}
Let $(\lambda_i,u_i)$ be an exact eigenpair for some $i\in J$. Then, for any $v_h\in V_h$, the following identity holds:
\[
a_h(\Lambda_hu_i,v_h)
=
\lambda_i b(P_hu_i,v_h).
\]
\end{lemma}

\begin{proof}
Since $\Lambda_h u_i\in W_h$, there exist coefficients $\{\alpha_m\}_{m\in J}$ such that
\[
\Lambda_h u_i=\sum_{m\in J}\alpha_m u_{h,m},
\qquad
\alpha_m=b(R_hu_i,u_{h,m}).
\]
By the symmetry of $a_h(\cdot,\cdot)$ and the definition of $R_h$, we obtain
\[
\alpha_m
=
b(R_hu_i,u_{h,m})
=
\lambda_{h,m}^{-1}a_h(R_hu_i,u_{h,m})=\lambda_{h,m}^{-1}
\lambda_i b(u_i,u_{h,m}).
\]
Consequently, 
\begin{align*}
a_h(\Lambda_hu_i,v_h)
&=
\sum_{m\in J}\alpha_m a_h(u_{h,m},v_h) =
\sum_{m\in J}\alpha_m\lambda_{h,m}b(u_{h,m},v_h) \\
&=
\lambda_i \sum_{m\in J}b(u_i,u_{h,m})b(u_{h,m},v_h) = \lambda_i b(P_hu_i,v_h),
\end{align*}
which completes the proof.
\end{proof}

The next lemma gives the $L^2$-control of the cluster projection error. It is a standard consequence of Assumption \ref{ass:sep} and the Aubin--Nitsche argument.
\begin{lemma}\label{lem:L2control}
For every $u_i\in W$ with $\|u_i\|_{L^2(\Omega)}^2=1$,
\begin{equation}\notag
\|\Lambda_hu_i-R_hu_i\|_{L^2(\Omega)}\le M_{h,J}\|u_i-R_hu_i\|_{L^2(\Omega)}
\end{equation}
and
\begin{equation}\notag
\|u_i-P_hu_i\|_{L^2(\Omega)}\le \|u_i-\Lambda_hu_i\|_{L^2(\Omega)}
\le (1+M_{h,J})\|u_i-R_hu_i\|_{L^2(\Omega)}.
\end{equation}
Moreover,
\begin{equation}\notag
\|u_i-P_h u_i\|_{L^2(\Omega)}\le C(1+M_{h,J})\rho_h\|u_i-\Lambda_h u_i\|_{h}.
\end{equation}
\end{lemma}
\begin{proof}
    The proof follows the same argument as Lemma 2.1 in \cite{Gallistl2015}.
\end{proof}

We now define the theoretical, non-computable error estimator. It will serve as an auxiliary estimator in the reliability analysis.
\begin{definition}
For each $T\in\mathcal T_h$, $e\in\mathcal E_h$ and $i \in J$, define
\begin{align*}
\mu_T(\lambda_i,u_i) & :=h_T^2\|\lambda_i P_h u_i-\Delta^2(\Lambda_hu_i)\|_{L^2(T)},\\
\mu_{e,1}(\lambda_i,u_i)& :=\sigma^{1/2}|e|^{-1/2}\|[\![\partial_{n}\Lambda_hu_i]\!]\|_{L^2(e)},\\
\mu_{e,2}(\lambda_i,u_i)& :=|e|^{1/2}\|[\![\partial_{nn}\Lambda_hu_i]\!]\|_{L^2(e)},\\
\mu_{e,3}(\lambda_i,u_i)& :=|e|^{3/2}\|[\![\partial_{n}(\Delta\Lambda_hu_i)]\!]\|_{L^2(e)}.
\end{align*}
The global theoretical estimator is denoted by
\[
\mu_h^2(\mathcal T_h,\lambda_i,u_i)
:=
\sum_{T\in\mathcal T_h}\mu_T^2(\lambda_i,u_i)
+
\sum_{e\in\mathcal E_h}\mu_{e,1}^2(\lambda_i,u_i)
+
\sum_{\ell=2,3}\sum_{e\in\mathcal E_h^I}\mu_{e,\ell}^2(\lambda_i,u_i).
\]
\end{definition}

Next, we derive a reliability bound for the theoretical estimator.
\begin{theorem}\label{rel}
There exists a constant $C>0$, independent of the mesh size and of the internal gaps within the target cluster, such that
\[
\|u_i-\Lambda_hu_i\|_h
\le
C\big(
\mu_h(\mathcal T_h,\lambda_i,u_i)
+
\lambda_i(1+M_{h,J})\rho_h\|u_i-\Lambda_h u_i\|_{h}\big),\qquad i\in J.
\]
\end{theorem}
\begin{proof}
Set
$e:=u_i-\Lambda_hu_i. $
We decompose the error into two parts using the enriching operator (recalling Lemma \ref{lem:enriching}):
\[
e=(u_i-E_h\Lambda_hu_i)+(E_h\Lambda_hu_i-\Lambda_hu_i)=: e_1+e_2.
\]
Hence
\[
\|e\|_h\le \|e_1\|_h+\|e_2\|_h.
\] 
By Lemma \ref{lem:enriching}, we get
\begin{equation}\label{e2}
    \|e_2\|_h
=
\|E_h\Lambda_hu_i-\Lambda_hu_i\|_h
\le
C\left(
\sum_{e\in\mathcal E_h}|e|^{-1}
\|[\![\partial_{n}\Lambda_hu_i]\!]\|_{L^2(e)}^2
\right)^{1/2}
\le
C\sum_{e\in\mathcal E_h}\mu_{e,1}.
\end{equation}

It remains to bound $\|e_1\|_h$. Since $e_1\in H_0^2(\Omega)$, we have $\|e_1\|_h=|e_1|_{H^2(\Omega)}.$ 
Let $\phi:=e_1,\ 
\phi_h:=\Pi_h\phi,
\ 
\eta:=\phi-\phi_h.
$ 
Using \eqref{eq:cont-eig} and Lemma \ref{lambda_h}, we obtain
\begin{equation}\label{e1}
    \begin{aligned}
       |e_1|_{H^2(\Omega)}^2
&=
\lambda_i b(u_i-P_hu_i,\phi)
+\Big(\lambda_i b(P_hu_i,\eta)-a_h(\Lambda_hu_i,\eta)\Big)
+ a_h(\Lambda_hu_i-E_h\Lambda_hu_i,\phi)\\
& =: I_1 +I_2 + I_3.
    \end{aligned}
\end{equation}
Combining the Poincar\'e inequality and Lemma \ref{lem:L2control}, the first term $I_1$ may be estimated as
\[
|\lambda_i b(u_i-P_hu_i,\phi)|
\le
C\lambda_i\|u_i-P_hu_i\|_{L^2(\Omega)}|\phi|_{H^2(\Omega)} \le C\lambda_i(1+M_{h,J})\rho_h\|u_i-\Lambda_h u_i\|_{h}|\phi|_{H^2(\Omega)}.
\]
For the second term $I_2$, we start from the alternative representation of the $C^0$IPG bilinear form. Since
$\Lambda_hu_i\in V_h$ is piecewise polynomial and $\eta\in H^2(\Omega,\mathcal T_h)\cap H_0^1(\Omega)$, by the integration by parts formula and the definition of $a_h(\cdot,\cdot)$, we have
\begin{equation}\label{ah2}
\begin{aligned}
a_h(\Lambda_hu_i,\eta)
={}&
\sum_{T\in\mathcal T_h}\int_T \Delta^2(\Lambda_hu_i)\,\eta\,dx
 +
\sum_{e\in\mathcal E_h}\int_e
\left\{\!\!\left\{\partial_{nn}\eta\right\}\!\!\right\}
[\![\partial_{n}\Lambda_hu_i]\!]\,ds
\nonumber\\
&+
\sum_{e\in\mathcal E_h^I}\int_e
[\![\partial_{n}(\Delta\Lambda_hu_i)]\!]\,\eta\,ds
 - \sum_{e\in\mathcal E_h^I}\int_e
[\![\partial_{nn}\Lambda_hu_i]\!]\,
\{\!\{\partial_{n}\eta\}\!\}\,ds
\nonumber\\
&-
\sum_{e\in\mathcal E_h^I}\int_e
[\![\partial_{nt}\Lambda_hu_i]\!]\,\partial_{t}\eta\,ds
 + \sigma\sum_{e\in\mathcal E_h}|e|^{-1}\int_e
[\![\partial_{n}\Lambda_hu_i]\!]\,[\![\partial_{n}\eta]\!]\,ds.
\notag
\end{aligned}
\end{equation}
For ease of note, we set $\psi_h=\lambda_i P_hu_i-\Delta^2(\Lambda_hu_i)$. Then, using \eqref{e1} and \eqref{ah2}, it holds that
\begin{equation}
\begin{aligned}
I_2
={}&
\sum_{T\in\mathcal T_h}\int_T \psi_h\eta\,dx
 -
\sum_{e\in\mathcal E_h}\int_e
\left\{\!\!\left\{\partial_{nn}\eta\right\}\!\!\right\}
[\![\partial_{n}\Lambda_hu_i]\!]\,ds
\nonumber\\
&-
\sum_{e\in\mathcal E_h^I}\int_e
[\![\partial_{n}(\Delta\Lambda_hu_i)]\!]\,\eta\,ds
+ \sum_{e\in\mathcal E_h^I}\int_e
[\![\partial_{nn}\Lambda_hu_i]\!]\,
\{\!\{\partial_{n}\eta\}\!\}\,ds
\nonumber\\
&+
\sum_{e\in\mathcal E_h^I}\int_e
[\![\partial_{nt}\Lambda_hu_i]\!]\,\partial_{t}\eta\,ds
- \sigma\sum_{e\in\mathcal E_h}|e|^{-1}\int_e
[\![\partial_{n}\Lambda_hu_i]\!]\,[\![\partial_{n}\eta]\!]\,ds .
\end{aligned}
  \label{R2}  
\end{equation}
Now we estimate the first term in \eqref{R2}. By the Cauchy--Schwarz inequality and Lemma \ref{lem:interp}, we obtain
\begin{align}\label{mut}
\left|
\sum_{T\in\mathcal T_h}\int_T \psi_h\eta\,dx
\right|
&\le
\sum_{T\in\mathcal T_h}
\|\lambda_i P_hu_i-\Delta^2(\Lambda_hu_i)\|_{L^2(T)}
\|\eta\|_{L^2(T)}
\nonumber\\
&\le
\left(
\sum_{T\in\mathcal T_h} h_T^4
\|\lambda_i P_hu_i-\Delta^2(\Lambda_hu_i)\|_{L^2(T)}^2
\right)^{1/2}
\left(
\sum_{T\in\mathcal T_h} h_T^{-4}\|\eta\|_{L^2(T)}^2
\right)^{1/2}
\nonumber\\
& \le
C\left(\sum_{T\in\mathcal T_h}\mu_T^2\right)^{1/2}
|\phi|_{H^2(\Omega)}.
\end{align}
Using similar techniques, the remaining terms may be estimated by the Cauchy--Schwarz inequality, the trace theorem and Lemma \ref{lem:interp} as follows: 
\begin{equation}\label{resteta}
    \begin{aligned}
  \left|
\sum_{e\in\mathcal E_h}\int_e
\left\{\!\!\left\{\partial_{nn}^2\eta\right\}\!\!\right\}
[\![\partial_{n}\Lambda_hu_i]\!]\,ds
\right| & \leq C\left(\sum_{e\in\mathcal E_h}\mu_{e,1}^2\right)^{1/2}
 |\phi|_{H^2(\Omega)}, \\
\left|
\sum_{e\in\mathcal E_h^I}\int_e
[\![\partial_{n}(\Delta\Lambda_hu_i)]\!]\,\eta\,ds
\right|
 & \le
C\left(
\sum_{e\in\mathcal E_h^I}\mu_{e,3}^2
\right)^{1/2}
|\phi|_{H^2(\Omega)},\\
\left|
\sum_{e\in\mathcal E_h^I}\int_e
[\![\partial_{nn}^2\Lambda_hu_i]\!]\,
\{\!\{\partial_{n}\eta\}\!\}\,ds
\right|
&\le
C\left(
\sum_{e\in\mathcal E_h^I}\mu_{e,2}^2
\right)^{1/2}
|\phi|_{H^2(\Omega)},\\
\left|
\sigma\sum_{e\in\mathcal E_h}|e|^{-1}\int_e
[\![\partial_{n}\Lambda_hu_i]\!]\,[\![\partial_{n}\eta]\!]\,ds
\right|
&\le
C\left(
\sum_{e\in\mathcal E_h}\mu_{e,1}^2
\right)^{1/2}
|\phi|_{H^2(\Omega)}.
\end{aligned}
\end{equation}
Combining \eqref{mut} and \eqref{resteta}, we get
\[
I_2 \leq \left|
\lambda_i b(P_hu_i,\eta)-a_h(\Lambda_hu_i,\eta)
\right|
\le
C\,
\mu_h(\mathcal T_h,\lambda_i,u_i)\,
|\phi|_{H^2(\Omega)}.
\]

For the third term $I_3$, using Lemma \ref{lem:enriching} and $\phi \in H_0^2(\Omega)$ yields that
\[
|a_h(\Lambda_hu_i-E_h\Lambda_hu_i,\phi)|
\le
C\|\Lambda_hu_i-E_h\Lambda_hu_i\|_h|\phi|_{h}
\le
C\,\mu_h(\mathcal T_h,\lambda_i,u_i)\,|\phi|_{H^2(\Omega)}.
\]
Substituting the above bounds for $I_1,\ I_2$ and $I_3$ into \eqref{e1}, we have
\[
|e_1|_{H^2(\Omega)}
\le
C\Big(
\mu_h(\mathcal T_h,\lambda_i,u_i)\Big)
+
C\lambda_i(1+M_{h,J})\rho_h\|u_i-\Lambda_h u_i\|_{h},
\]
which, together with \eqref{e2}, completes the proof.
\end{proof}

We now turn to the efficiency of the theoretical error estimator.

\begin{theorem}\label{thm:efficiency}
There exists a constant $C$ such that
\[
\mu_h(\mathcal T_h,\lambda_i,u_i)
\le
C\left(
1+\lambda_i(1+M_{h,J})\rho_h
\right)
\|u_i-\Lambda_hu_i\|_h,\qquad i \in J.
\]
\end{theorem}
\begin{proof}
The proof relies on standard bubble-function arguments applied separately to the element residual terms and the edge jump terms (see \cite{1994A,Gallistl2015, Brenner2012}).
\end{proof}

\subsection{Equivalence of error estimators}\label{sec:equivalence}
In this subsection, we show that the theoretical estimator is locally equivalent to the computable estimator. This equivalence plays a key role in proving Theorem \ref{thm:main-reliability}. We first recall a basis property of $\{P_hu_i\}_{i\in J}$ and $\{\Lambda_hu_i\}_{i\in J}$, which follows from Lemma 5.1 of \cite{Gallistl2015}.

\begin{lemma}\label{lem:gallistl-basis}
Assume that
\[
\varepsilon:=\max_{i\in J}\|u_i-\Lambda_hu_i\|_{L^2(\Omega)}
\le \sqrt{1+\frac{1}{2N}}-1,
\qquad N=\operatorname{card}(J).
\]
Then both $\{P_hu_i\}_{i\in J}$ and $\{\Lambda_hu_i\}_{i\in J}$ form a basis of $W_h$. Moreover, for any $w_h\in W_h$ with $\|w_h\|_{L^2(\Omega)}=1$, if
\[
w_h=\sum_{i\in J}\beta_i P_hu_i
\qquad\text{and}\qquad
w_h=\sum_{i\in J}\gamma_i \Lambda_hu_i,
\]
then
\[
\max\left\{\sum_{i\in J}|\beta_i|^2,\ \sum_{i\in J}|\gamma_i|^2\right\}
\le 2+4N.
\]
\end{lemma}

We now use this basis property to derive the local equivalence of the two estimators.
\begin{theorem}[Local equivalence]\label{thm:local-equivalence}
Assume that the exact cluster eigenvalues $\{\lambda_i\}_{i\in J}$ and the discrete cluster eigenvalues $\{\lambda_{h,j}\}_{j\in J}$ are contained in a common compact interval
$[\lambda_m,\lambda_{M+1}]$
and that the assumptions of Lemma~\ref{lem:gallistl-basis} are satisfied. Then, for any $T\in\mathcal T_h$,
{
\begin{equation}\label{5.6}
   N^{-1}\left(\frac{\lambda_{m}}{\lambda_{M+1}}\right)^2\sum_{i\in J}\mu_h^2(T,\lambda_i,u_i)
\le
 \eta_h^2(T)
\le
(2N+4N^2)\sum_{i\in J}\mu_h^2(T,\lambda_i,u_i), 
\end{equation}
}
where the local computable estimator is given by
\[
\eta_h^2(T)
:=
\sum_{j\in J}
\left(
\eta_T^2(\lambda_{h,j},u_{h,j})
+
\sum_{e\subset\partial T}
\eta_{e,1}^2(\lambda_{h,j},u_{h,j})
+
\sum_{\ell=2,3}\sum_{e\subset\partial T\cap\mathcal E_h^I}
\eta_{e,\ell}^2(\lambda_{h,j},u_{h,j})
\right),
\]
the local theoretical estimator is given by
\[
\mu_h^2(T,\lambda_i,u_i)
:=
\mu_T^2(\lambda_i,u_i)
+
\sum_{e\subset\partial T}
\mu_{e,1}^2(\lambda_i,u_i)
+
\sum_{\ell=2,3}\sum_{e\subset\partial T\cap\mathcal E_h^I}
\mu_{e,\ell}^2(\lambda_i,u_i),
\]
and $N=\operatorname{card}(J)$.
\end{theorem}
\begin{proof}
For clarity, we present the details for both directions. Throughout the proof, $\|\cdot\|_T$ and $\|\cdot\|_e$ denote the $L^2(T)$ and $L^2(e)$ norms, respectively. For each $i\in J$, expand
\begin{equation}\label{5.1}
    P_hu_i=\sum_{j\in J}\alpha_{ij}u_{h,j},
\qquad
\Lambda_hu_i=\sum_{j\in J}\nu_{ij}u_{h,j}.
\end{equation}
Since $\alpha_{ij}=b(u_i,u_{h,j}), $ we have $\sum_{j\in J}\alpha_{ij}^2=\|P_hu_i\|_{L^2(\Omega)}^2\le 1. $ Moreover, it holds that
\[
\nu_{ij}
= b(R_hu_i,u_{h,j})
=
\frac{\lambda_i}{\lambda_{h,j}}\alpha_{ij},
\]
and therefore
\begin{equation}
    \sum_{j\in J}\nu_{ij}^2 =
\sum_{j\in J}\left(\frac{\lambda_i}{\lambda_{h,j}}\right)^2\alpha_{ij}^2
\le
\left(\frac{\lambda_{M+1}}{\lambda_{m}}\right)^2.
\label{5.3}
\end{equation}
Using \eqref{5.1}-\eqref{5.3} and the Cauchy--Schwarz inequality, we obtain
\begin{equation*}
\begin{aligned}
\mu_{T}^2(\lambda_i,u_i)
&=
h_T^4\left\|
\sum_{j\in J}\nu_{ij}\left(\lambda_{h,j}u_{h,j}-\Delta^2u_{h,j}\right)
\right\|_T^2\\
&\le
\left(\sum_{j\in J}\nu_{ij}^2\right)
\left(
\sum_{j\in J}h_T^4\|\lambda_{h,j}u_{h,j}-\Delta^2u_{h,j}\|_T^2
\right)\\
&\le
\left(\frac{\lambda_{M+1}}{\lambda_{m}}\right)^2 \sum_{j\in J} \eta_{T}^2(\lambda_{h,j},u_{h,j}).
\end{aligned}  
\end{equation*}
The edge terms are estimated in the same way. Indeed, since the jump operators are linear, by \eqref{5.3} and the Cauchy--Schwarz inequality, it holds that 
\begin{align*}
\mu_{e,1}^2(\lambda_i,u_i)
&=
\sigma h_e^{-1}
\left\|
\sum_{j\in J}\nu_{ij}
[\![\partial_{n} u_{h,j}]\!]
\right\|_e^2
\\
&\le
\left(\sum_{j\in J}\nu_{ij}^2\right)
\sum_{j\in J}
\sigma h_e^{-1}
\|[\![\partial_{n} u_{h,j}]\!]\|_e^2
\\
&\le
\left(\frac{\lambda_{M+1}}{\lambda_m}\right)^2
\sum_{j\in J}
\eta_{e,1}^2(\lambda_{h,j},u_{h,j}).
\end{align*}
Similarly, for the remaining edge terms, we obtain
\[
\mu_{e,2}^2(\lambda_i,u_i)
\le
\left(\frac{\lambda_{M+1}}{\lambda_m}\right)^2
\sum_{j\in J}
\eta_{e,2}^2(\lambda_{h,j},u_{h,j}),
\]
\[
\mu_{e,3}^2(\lambda_i,u_i)
\le
\left(\frac{\lambda_{M+1}}{\lambda_m}\right)^2
\sum_{j\in J}
\eta_{e,3}^2(\lambda_{h,j},u_{h,j}).
\]
Combining the element and edge estimates gives
\[
\mu_h^2(T,\lambda_i,u_i)
\le
\left(\frac{\lambda_{M+1}}{\lambda_m}\right)^2
\eta_h^2(T).
\]
Summing over all $i\in J$ and rearranging the inequality, we obtain
\begin{equation}\label{5.8}
    N^{-1}
\left(\frac{\lambda_m}{\lambda_{M+1}}\right)^2
\sum_{i\in J}
\mu_h^2(T,\lambda_i,u_i)
\le
\eta_h^2(T).
\end{equation}
It remains to prove the reverse estimate. By Lemma~\ref{lem:gallistl-basis}, for each $j\in J$, there exist coefficients $\{\beta_{ji}\}_{i\in J}$ and $\{\gamma_{ji}\}_{i\in J}$ such that
\[
u_{h,j}
=
\sum_{i\in J}\beta_{ji}P_hu_i
=
\sum_{i\in J}\gamma_{ji}\Lambda_hu_i,
\]
with
\[
\sum_{i\in J}\beta_{ji}^2\le 2+4N,
\qquad
\sum_{i\in J}\gamma_{ji}^2\le 2+4N.
\]
Using \eqref{5.1} and comparing the coefficients of $u_{h,k}$ in the two representations of $u_{h,j}$, we have
\[
\sum_{i\in J}\beta_{ji}\alpha_{ik}=\delta_{jk},
\qquad
\sum_{i\in J}\gamma_{ji}\nu_{ik}=\delta_{jk}.
\]
Since $\nu_{ik}=(\lambda_i/\lambda_{h,k})\alpha_{ik}$, it follows that
\[
\sum_{i\in J}\gamma_{ji}\lambda_i\alpha_{ik}
=
\lambda_{h,k}\delta_{jk}
=
\lambda_{h,j}\delta_{jk}
=
\sum_{i\in J}\lambda_{h,j}\beta_{ji}\alpha_{ik}.
\]
Hence,
\[
\sum_{i\in J}
\left(
\gamma_{ji}\lambda_i-\lambda_{h,j}\beta_{ji}
\right)\alpha_{ik}=0
\qquad \forall k\in J.
\]
Since $\{P_hu_i\}_{i\in J}$ is a basis of $W_h$, the matrix $(\alpha_{ik})_{i,k\in J}$ is nonsingular. Hence
\[
\lambda_{h,j}\beta_{ji}=\lambda_i\gamma_{ji},
\qquad i\in J.
\]
Using this coefficient relation, we have
\begin{align*}
\lambda_{h,j}u_{h,j}-\Delta^2u_{h,j}
&=
\lambda_{h,j}
\sum_{i\in J}\beta_{ji}P_hu_i
-
\Delta^2
\sum_{i\in J}\gamma_{ji}\Lambda_hu_i
\\
&=
\sum_{i\in J}\gamma_{ji}
\left(
\lambda_iP_hu_i-\Delta^2\Lambda_hu_i
\right).
\end{align*}
Therefore, by the Cauchy--Schwarz inequality and the coefficient bound,
\begin{align*}
\eta_T^2(\lambda_{h,j},u_{h,j})
&=
h_T^4
\|\lambda_{h,j}u_{h,j}-\Delta^2u_{h,j}\|_T^2
\\
&=
h_T^4
\left\|
\sum_{i\in J}\gamma_{ji}
\left(
\lambda_iP_hu_i-\Delta^2\Lambda_hu_i
\right)
\right\|_T^2
\\
&\le
\left(
\sum_{i\in J}\gamma_{ji}^2
\right)
\left(
\sum_{i\in J}
h_T^4
\|\lambda_iP_hu_i-\Delta^2\Lambda_hu_i\|_T^2
\right)
\\
&\le
(2+4N)
\sum_{i\in J}
\mu_T^2(\lambda_i,u_i).
\end{align*}
Using the same techniques to estimate the edge terms, it follows that
\begin{align*}
    \eta_{e,1}^2(\lambda_{h,j},u_{h,j})
& \le
(2+4N)
\sum_{i\in J}
\mu_{e,1}^2(\lambda_i,u_i),\\
\eta_{e,2}^2(\lambda_{h,j},u_{h,j})
& \le
(2+4N)
\sum_{i\in J}
\mu_{e,2}^2(\lambda_i,u_i),\\
\eta_{e,3}^2(\lambda_{h,j},u_{h,j})
&\le
(2+4N)
\sum_{i\in J}
\mu_{e,3}^2(\lambda_i,u_i).
\end{align*}
Combining the element and edge estimates and summing over all $j\in J$ yields
\[
\eta_h^2(T)
\le
(2N+4N^2)
\sum_{i\in J}
\mu_h^2(T,\lambda_i,u_i),
\]
which, together with \eqref{5.8}, completes the proof.
\end{proof} 

Combining the results of Sections \ref{3.1} and \ref{sec:equivalence}, we obtain the proof of the main theorem.\\
{\bf{Proof of Theorem \ref{thm:main-reliability}:}}
Summing the local equivalence estimate \eqref{5.6} over all $T\in\mathcal T_h$ and combining it with Theorem \ref{rel}, we obtain
\begin{equation*}
    \begin{aligned}
        \|u_i-\Lambda_hu_i\|_h &\le C\Big(\mu_h(\mathcal{T}_h,\lambda_i,u_i)+ \lambda_i(1+M_{h,J})\rho_h \|u_i-\Lambda_h u_i\|_{h}\Big) \\
& \le C\left(\left(\frac{\lambda_{M+1}}{\lambda_{m}}\right)\eta_h
+\lambda_i(1+M_{h,J})\rho_h\|u_i-\Lambda_hu_i\|_{h}\right).
    \end{aligned}
\end{equation*}
 The second estimate follows similarly from \eqref{5.6} and Theorem \ref{thm:efficiency}, which completes the proof of Theorem \ref{thm:main-reliability}.\qed

\section{Numerical experiments}
This section presents two numerical examples in an L-shaped domain and a crack domain. We use the $C^0$IPG method with polynomial degree $k=2$ and a fixed penalty parameter $\sigma=10$. For both examples, the initial mesh is taken to be a uniform triangulation of the domain. We adopt the D\"orfler marking strategy and refine the marked elements by the newest vertex bisection. Here, $ d.o.f. $ denotes the number of degrees of freedom.
\subsection{L-shaped domain}
\begin{example}
We solve \eqref{eq:strong-eig} in the L-shaped domain 
$\Omega=(-1,1)^2\setminus([0,1]\times[-1,0])$. 
We focus on the eigenvalues indexed by $J=\{15,16\}$ and their corresponding eigenfunctions. These two eigenvalues $\lambda_{15}$ and $\lambda_{16}$ are close to each other and form a representative eigenvalue cluster in the L-shaped domain. The reliability and efficiency of the error estimator $\eta_h$ are then tested using the D\"orfler marking strategy with marking parameter $\theta=0.5$.
\end{example}

\begin{figure}[H]
\centering
\begin{minipage}{0.41\textwidth}
    \centering
    \includegraphics[width=\textwidth]{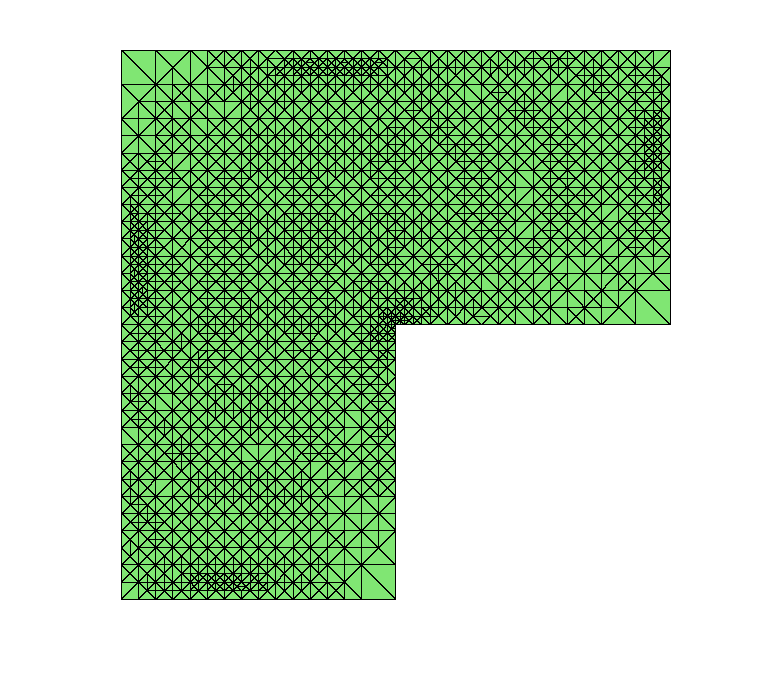}
    \caption{Adaptive mesh with 14879 degrees of freedom for the L-shaped domain.}\label{fig:lshapemesh}
\end{minipage}
\hfill
\begin{minipage}{0.45\textwidth}
    \centering
    \includegraphics[width=\textwidth]{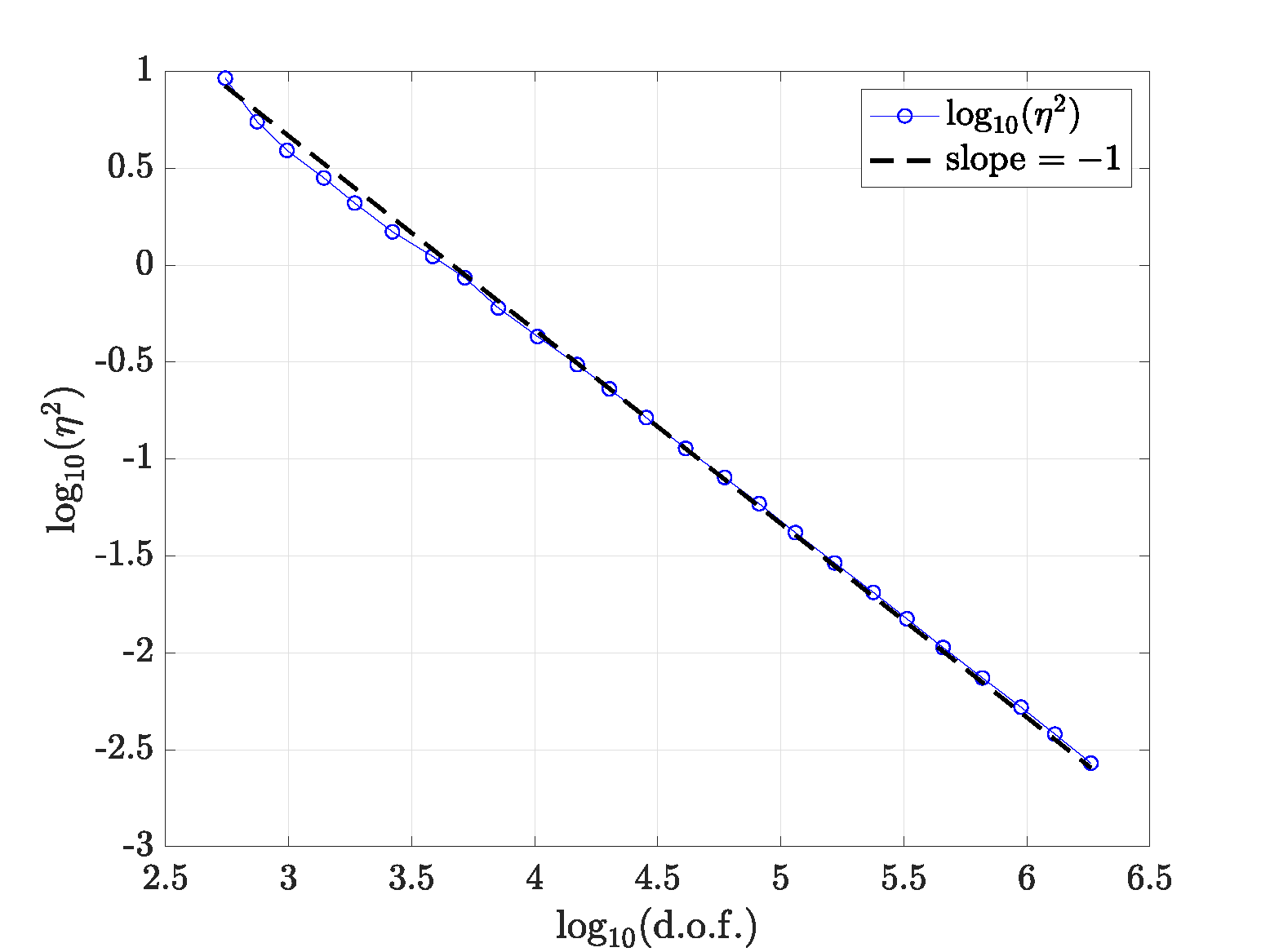}
    \caption{A posteriori estimator for $J = \{15,\ 16\}$.}\label{fig:lshapeeta}
\end{minipage}
\end{figure}

\par Figure \ref{fig:lshapemesh} displays the adaptive mesh with 14879 degrees of freedom, where the refinement is concentrated around the corner of the L-shaped domain, indicating that the proposed estimator effectively captures the singularity. In Figure \ref{fig:lshapeeta}, the reference line has slope $-1$. We observe that the curve of the squared estimator $\eta_h^2$ is parallel to the reference line, which verifies our theoretical findings, i.e., the reliability and efficiency of the estimator. Moreover, numerical results show that the estimator $\eta_h^2$ converges at the optimal rate $O( d.o.f. ^{-1})$.

\subsection{Perturbed symmetric crack domain}
\par In this subsection, we present numerical results for a perturbed symmetric crack domain. 
\begin{example}
We solve \eqref{eq:strong-eig} in a square with four perturbed symmetric slits defined by
\[
\Omega=(-1,1)^2 \backslash\binom{\operatorname{conv}\{(0.5005,0),(1,0)\} \cup \operatorname{conv}\{(0,0.501),(0,1)\}}{\cup \operatorname{conv}\{(-0.499,0),(-1,0)\} \cup \operatorname{conv}\{(0,-0.5),(0,-1)\}} .
\]
From the adaptive computation, it may be seen that $\lambda_3$ and $\lambda_4$ are close to each other. We therefore compute the eigenvalues indexed by $J=\{3,4\}$ and the corresponding eigenfunctions. The D\"orfler marking strategy is used with marking parameter $\theta=0.1$.
\end{example}

Figure \ref{fig:crackmesh} shows the adaptive mesh with 28453 degrees of freedom. We observe that the refinement is mainly concentrated around the crack, which indicates that the proposed estimator effectively captures the singularities of the eigenfunctions. Figure \ref{fig:cracketa} plots the squared estimator $\eta_h^2$, together with a reference line of slope $-1$. The curve of $\eta_h^2$ is parallel to the reference line, which is consistent with the theoretical results established in the previous sections. In addition, this behavior indicates the optimal convergence rate $O( d.o.f. ^{-1})$.

\begin{figure}[H]
\centering
\begin{minipage}{0.48\textwidth}
    \centering
    \includegraphics[width=\textwidth]{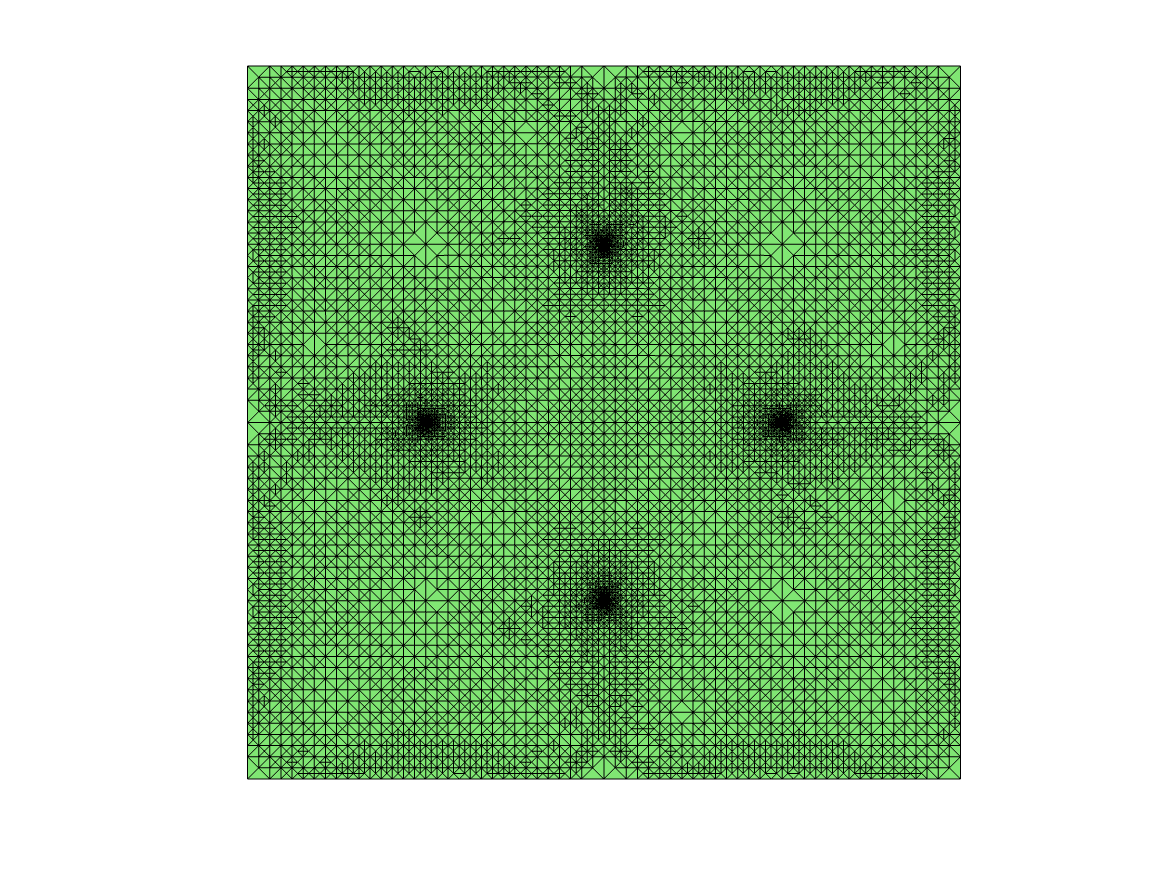}
    \caption{Adaptive mesh with 28453 degrees of freedom for the crack domain.}\label{fig:crackmesh}
\end{minipage}
\hfill
\begin{minipage}{0.45\textwidth}
    \centering
    \includegraphics[width=\textwidth]{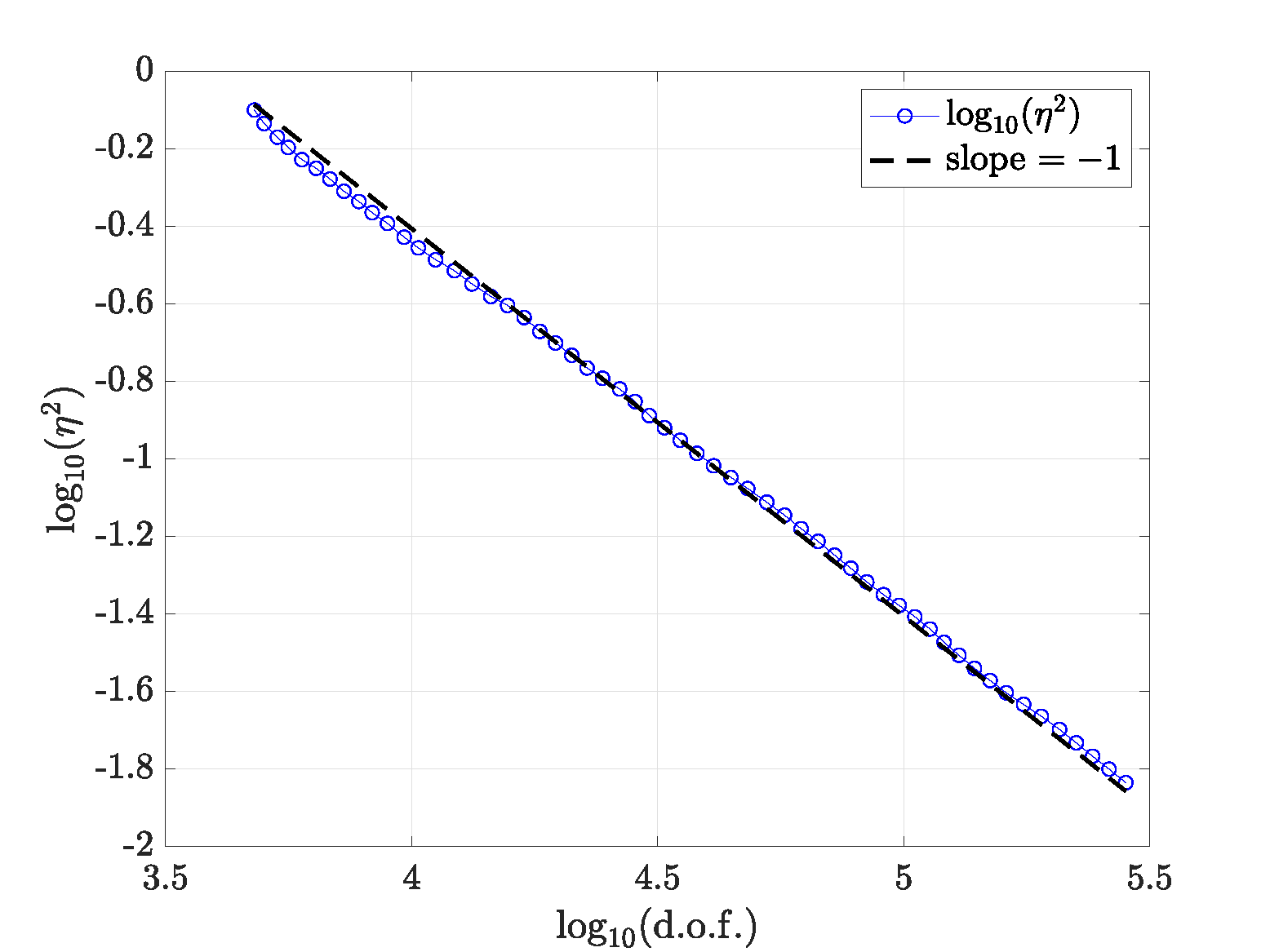}
    \caption{A posteriori estimator for $J = \{3,\ 4\}$.}\label{fig:cracketa}
\end{minipage}
\end{figure}

\begin{figure}[H]
\centering
\begin{minipage}{0.45\textwidth}
    \centering
    \includegraphics[width=\textwidth]{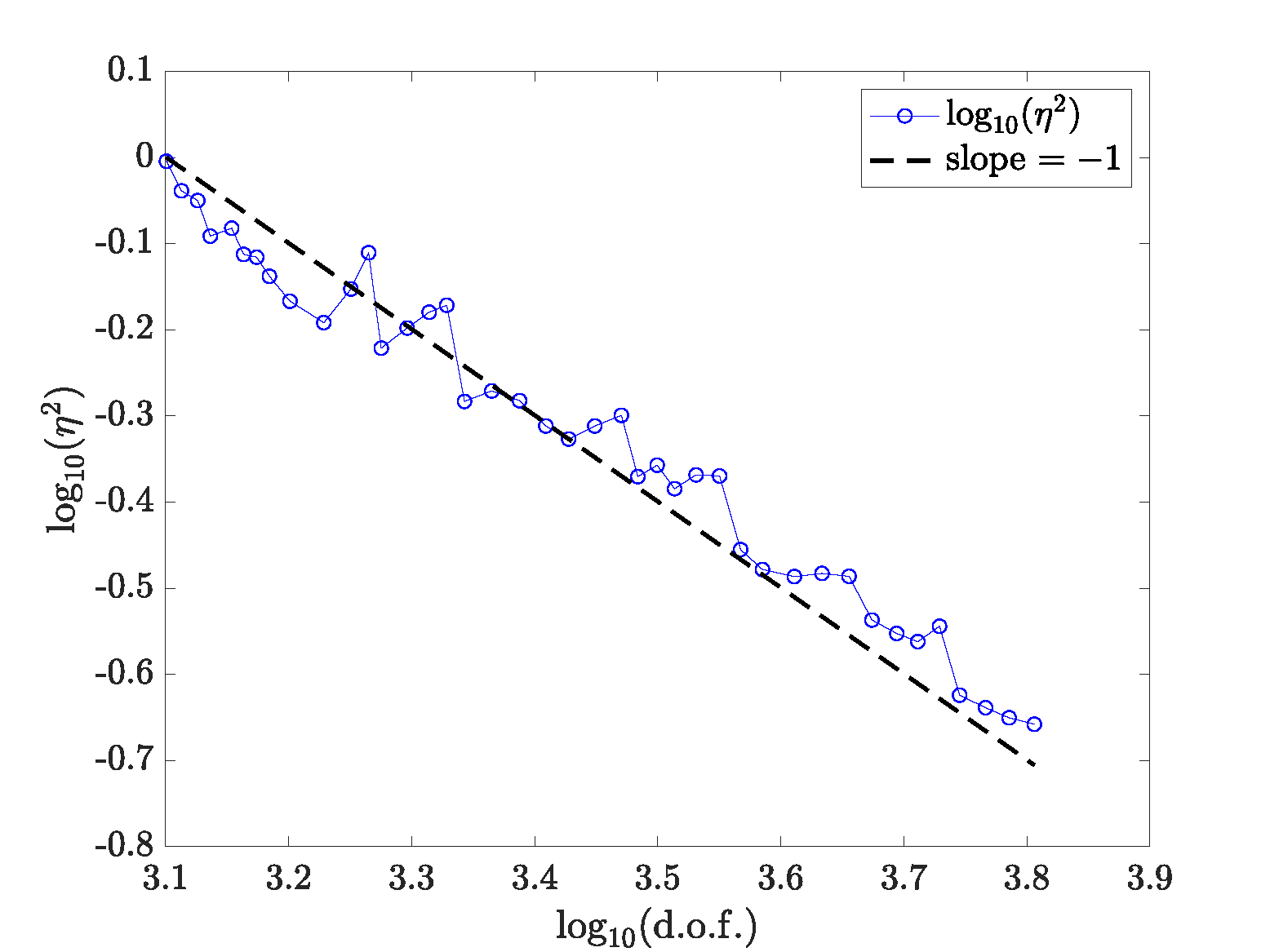}
    \caption{A posteriori estimator for $J = \{3\}$.}\label{fig:crack3}
\end{minipage}
 \hfill
\begin{minipage}{0.45\textwidth}
    \centering
    \includegraphics[width=\textwidth]{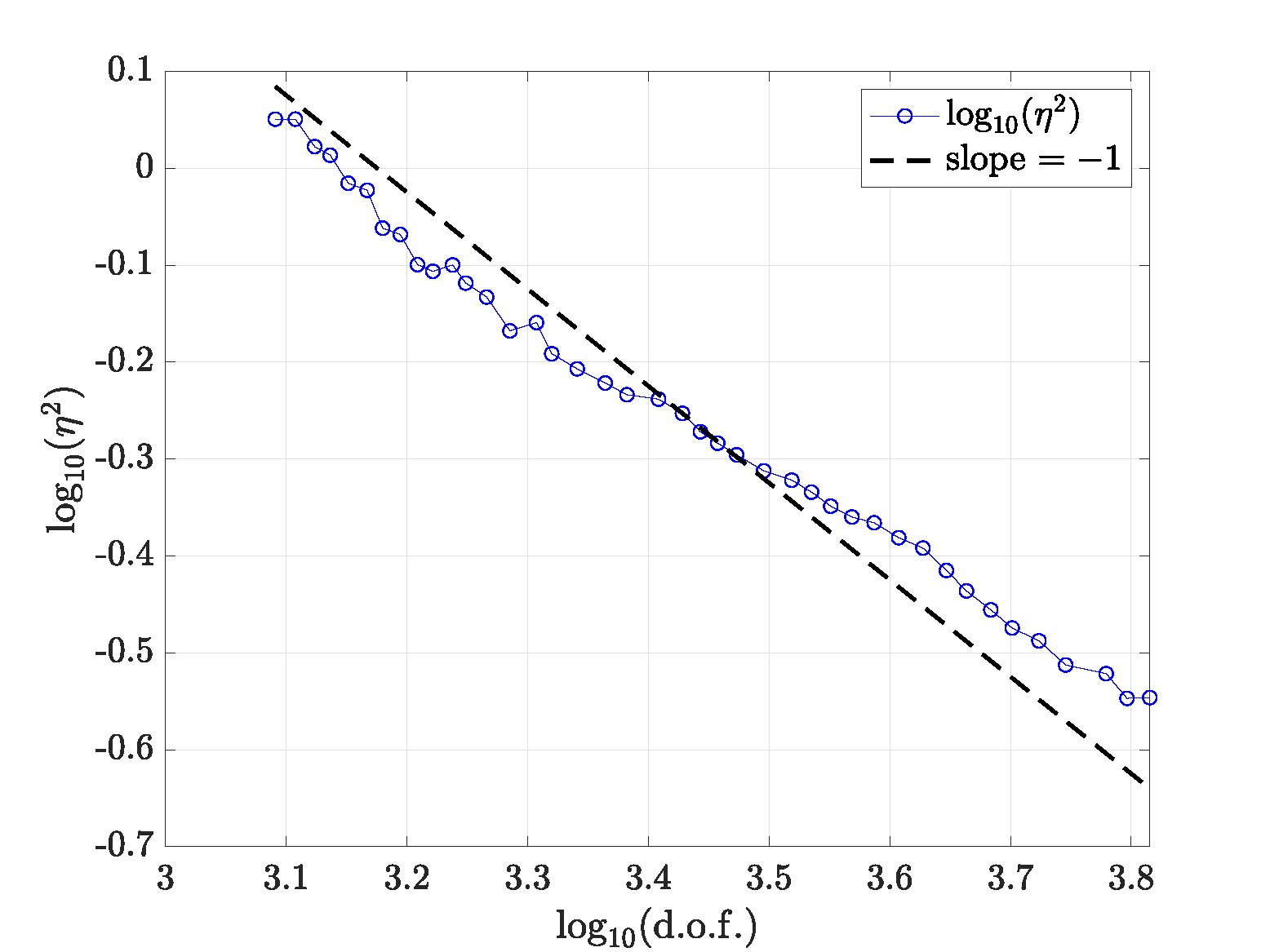}
    \caption{A posteriori estimator for $J = \{4\}$.}\label{fig:cracketa4}
\end{minipage}
\end{figure}

Furthermore, using the same initial mesh $\mathcal{T}_0$ and the same marking parameter $\theta=0.1$, we compare the estimator \eqref{compueta} corresponding to the index sets $J=\{3,4\}$, $\{3\}$, and $\{4\}$. In separate computations for $\{3\}$ and $\{4\}$, the estimators show unstable oscillations, as seen in Figures \ref{fig:crack3} and \ref{fig:cracketa4}. This suggests that the cluster estimator $\eta_{h,J}$ provides a more stable indicator for clustered eigenvalues.

\section{Conclusion}\label{sec:conclusion}

In this paper, we propose and analyze a residual-type a posteriori error estimator for the $C^0$ interior penalty Galerkin approximation of multiple and clustered eigenvalues of the biharmonic operator. The analysis is based on a cluster projection framework and a theoretical error estimator associated with the exact eigenspace. It is proven that the theoretical estimator is reliable and efficient in the mesh-dependent energy norm. Moreover, its equivalence with the computable estimator is established, which leads to the reliability and efficiency of the proposed estimator.

A key feature of the analysis is that the reliability and efficiency bounds are independent of the pairwise distances between eigenvalues within the cluster. This makes the estimator suitable for tightly clustered eigenvalues. Numerical experiments in the L-shaped domain and the crack domain verify the theoretical findings and show the stability of the proposed estimator.

\backmatter

\section*{Declarations}
\begin{itemize}
\item {\bf Funding} The second author is supported by National Natural Science Foundation of China (Grant No. 12401481).
\item {\bf Conflict of interest/Competing interests} There is no potential conflict of interest.
\item {\bf Ethics approval and consent to participate} This article does not contain any studies involving animals or human participants.
\item {\bf Consent for publication} Not applicable.
\item {\bf Data availability} This research does not use any external or author-collected data.
\item {\bf Materials availability} Not applicable.
\item {\bf Code availability} Not applicable.
\item {\bf Author contribution} Jianing Guo: theoretical derivation, numerical experiments, writing original draft preparation. Qigang Liang: theoretical derivation, conceptualization, methodology, supervision, writing—reviewing and final approval.
\end{itemize}

\bibliography{sn-bibliography}
\end{document}